\documentclass[10pt]{amsart}
\usepackage{amsfonts,amssymb,amsmath, forest}
\usepackage{amssymb, amsmath}
\usepackage{float,epsfig}
\usepackage{xcolor}
\usepackage{ mathrsfs }
\usepackage{graphicx}
\usepackage{hyperref}
\usepackage{mathabx}
\usepackage[left=2.8cm,top=1 cm,right=2.8cm,nofoot]{geometry}

\usepackage{cleveref}
\usepackage[mathlines]{lineno}
\usepackage{mathtools}

\DeclareMathOperator*{\essinf}{ess-inf}
\DeclareMathOperator*{\rank}{rank}
\DeclareMathOperator*{\range}{range}
\DeclareMathOperator*{\Ker}{Ker}
\DeclareMathOperator*{\dom}{dom}

\newcommand{\ol}{\overline}
\DeclareMathOperator*{\Span}{span}

\newtheorem{thm}{Theorem}[section]

\newtheorem{lem}[thm]{Lemma}

\newtheorem{deff}[thm]{Definition}

\newtheorem{prop}[thm]{Proposition}
\newtheorem{cor}[thm]{Corollary}

\newenvironment{defn}{\begin{deff}
		\rm }{\end{deff}}
\newcommand{\mc}{\mathcal}

\renewcommand{\ss}{\subseteq}
\newcommand{\ra}{\rightarrow}
\newcommand{\msc}{\mathscr}

\newcommand{\lan}{\langle}
\newcommand{\ran}{\rangle}

\def \l2x{L^2(X;\mc H)}
\def \l2w{L^2(X;\mc H)}

\def \dx{{d_{\mu_{X}}(x)}}
\def \dm{{d_{\mu_{\mc M}}(s)}}

\def \sd{{S_{\mc D}(\msc A)}}

\def \domega{{d_{\mu_X}(x)}}
\def \ds{{d_{\mu_{\mc M}}(s)}}
\def \dgamma{{d_{\mu_{\Gamma}}(\gamma)}}
\newcommand{\blue}{\textcolor{blue}}

\theoremstyle{definition}

\theoremstyle{remark}

\numberwithin{equation}{section}

\usepackage{scalerel,stackengine}
\stackMath
\newcommand\reallywidehat[1]{%
	\savestack{\tmpbox}{\stretchto{%
			\scaleto{%
				\scalerel*[\widthof{\ensuremath{#1}}]{\kern.1pt\mathchar"0362\kern.1pt}%
				{\rule{0ex}{\textheight}}%WIDTH-LIMITED CIRCUMFLEX
			}{\textheight}% 
		}{2.4ex}}%
	\stackon[-6.9pt]{#1}{\tmpbox}%
}
\begin{document}
	\title[Dual frames   using  Infimum cosine angle]{ A Characterization of MG   Dual frames \\  using  Infimum cosine angle}

\author{Sudipta Sarkar}
\address{Department  of Mathematics,
	Indian Institute of Technology Indore,
	Simrol, Khandwa Road,
	Indore-453 552}
\email{sudipta.math7@gmail.com, nirajshukla@iiti.ac.in}

\thanks{Research of S. Sarkar and N. K. Shukla was supported by   research grant from CSIR, New Delhi [09/1022(0037)/2017-EMR-I] and NBHM-DAE [02011/19/2018-NBHM(R.P.)/R\&D II/14723], respectively.}

\author{Niraj K. Shukla}


%%%%%%%%%%%%%%%%%%%%%%%%%%%%%%%%%%%%%%%%%%%%%%%%%%%%%%%
%    General info


%%%%%%%%%%%%%%%%%%%%%%%%%%%%%%%%%%%%%%%%%%%%%%

\begin{abstract}
	 This article discusses the construction of  dual frames and their uniqueness for the multiplication generated frames on $L^2(X; \mathcal H)$, where $X$ is a $\sigma$-finite measure. A necessary and sufficient condition of such duals associated to infimum cosine angle  is obtained.    The result is illustrated for the translation-generated systems on a locally compact group (not necessarily abelian ) by action of its abelian subgroup.
	 \end{abstract}

\maketitle
%\linenumbers
\section{{\bf Introduction }}\label{S:Introduction}
 Let   $\mc H$ be a separable Hilbert space,   and $(X, \mu_X)$  be a  $\sigma$-finite   measure space so that $L^2(X)$ is separable.  The   vector-valued space  $L^2(X;\mathcal H)$ defined by
	$$
	L^2(X;\mathcal H)=\left\{ \text{measaurable} \ f : X \rightarrow \mathcal H \ | \   \|f \|^2=   \int_{X}\|f(x)\|_{\mc H}^2 \ \dx <\infty \right\},
	$$
	is a Hilbert space with the inner product  
	$\langle f, g \rangle=\int_{X} \langle f(x), g(x) \rangle\ \dx$,  for $f, g \in L^2(X;\mathcal H).$  We define a multiplication invariant space   on $L^2(X;\mathcal H)$.
	\begin{defn} \label{MIOpDefn}Let  $\msc V$ be a closed subspace of  $L^2(X; \mc H)$ and $\mc D \subset L^\infty(X)$.  The space $\msc V$ is said to be  \textit{multiplication invariant} (MI)   corresponding to $\mc D$ if  $M_\phi f \in \msc V$ for all $\phi \in \mc D$ and $f \in \msc V$, where   the \textit{multiplication operator} $M_\phi$ on  $L^2(X;\mathcal H)$ is defined as follows:
		$$
		(M_\phi f)(x)=\phi(x) f(x), \ a.e. \  x \in X, \ f \in  L^2(X;\mathcal H).	$$
	\end{defn}
The operator    $M_\phi$ is   bounded  and linear    satisfying $\|M_\phi\|=\|\phi\|_{L^\infty}$ when $X$  is a $\sigma$-finite measure space. 
	  
	Bownik and Ross in \cite{bownik2015structure} studied the structure of MI spaces in $L^2(X; \mc H)$ and classified them using range functions $J$, where $J$ is a mapping from $X$ to the collection of closed subspaces of $\mc H$.  Further,   Iverson in \cite{iverson2015subspaces}, and Bownik and  Iverson in \cite{BownikIversion2021} investigated the frame properties of multiplication generated systems associated with the Parseval determining set. 
	A set $\mc D=\{g_s: s \in \mc M\}$ in $L^{\infty}(X)$ is known as   \textit{Parseval determining set} for $L^1(X)$ if   for each $f \in L^1(X)$,   $s\mapsto \int_{X}f(x)\overline{g_s (x)}d_{\mu_{X}}(x)$  is measurable on $\mc M$,  and  
	\begin{align}\label{Parsevaldeter}
	\int_{\mc M}\mid\int_{X}f(x)\overline{g_s(x)}d_{\mu_{X}}(x)\mid ^2d_{\mu_{\mc M}}(s)=\int_{X} | f(x) |^2 d_{\mu_{X}}(x),
	\end{align}
	where   $(\mc M,\mu_{\mc M})$ is a $\sigma$-finite measure space. 	The  \textit{multiplication generated (MG)}  system corresponding  to the Parseval determining set $\mc D$ is   given by,
	$
	\mc E_{\mc  D}(\msc A) :=\left\{M_{\phi_s}f_i (\cdot)=\phi_s(\cdot)f_i (\cdot): s\in \mc M, i=1,2, \dots, r\right\},
	$ 
	for a collection of functions $\msc A =\{f_i\}_{i=1}^r$ in $L^2 (X; \mc H)$. 
We denote  $S_{\mc D}(\msc A) :=\ol{\mbox{span}} \, \mc E_{\mc D}(\msc A)$ and 
the associated range function  $J_{\msc A}(x):=\Span\{f_i(x): i=1,2, \dots,r\}$, for a.e. $x\in X$ \cite[Proposition 2.2 (iii)]{iverson2015subspaces}.

%	To classify all the $\mc  D$-multiplication invariant space   Bownik and Ross started with the concept of range function in \cite{bownik2015structure}.
%	A range function is a map 
%	$$J: X\ra \{\mbox{closed subspaces of} \ \mc H\}.$$
%	
%	There is a 1-1 correspondence between MI spaces and measurable range functions. For the class $\mc S_{\mc D}(\msc A)$ the associated range function is $J_{\msc A}(x)=\Span \{\varphi(x):\varphi \in \msc A\}$ \cite[Theorem 2.4]{bownik2015structure}.
%	
%	
%	
%	 Further,  Bownik and Iversion in \cite{BownikIversion2021} studied the MI operators  acting on subspaces  of  $L^2(X; \mc H)$ and established several results about frames for    multiplications generated systems  in $L^2(X; \mc H)$. 
%	
	In this chain of research our goal is to find characterization  results for  the MG  duals $\mc E_{\mc D}(\msc A')$ of a frame $\mc E_{\mc D}(\msc A)$   associated with the infimum cosine angles  between the closed subspaces  $\mc S_{\mc D}(\msc A')$ and $\mc S_{\mc D}(\msc A)$  of $L^2 (X; \mc H )$, for some finite  collection of functions $\msc A' $ in $L^2 (X; \mc H)$. 
	The  infimum cosine angle between two closed subspaces of Hilbert spaces \cite{aldroubi1998construction} is defined as  follows:
	\begin{defn}\label{D-InfCosine} Let $  V$ and $W$ be closed subspaces of $\mc H$. The
		\textit{infimum cosine angle} between $V$ and $ W$  of $\mc H$ is defined by
		$$R(V,  W)=\smashoperator{\inf_{v\in  V\backslash\{0\}}} \frac{\|P_{ W}v\|}{\|v\|},$$ where $P_W$ is the projection on $W$.
	\end{defn}
 In general,  $R(V,W)\neq R(W,V)$. If $R(V, W)>0$ and $R(W, V)>0$ then $R(V, W)=R(W,V)$,  and hence  we can decompose the Hilbert space $\mc H=V\, {\oplus} \, W^\perp$  (not necessary orthogonal direct sum), means, $\mc H= V+ W^\perp$ and $ V\bigcap  W^{\perp}=0$ \cite{christensen2004oblique}. In addition, if   the following reproducing formula holds : 
   $$f=\sum_{k\in I} \langle f, f_k\rangle g_k,  \ \forall f\in V,$$
    where $\{f_k\}_{k\in I}$ and $\{g_k\}_{k\in I}$ are Bessel sequences in $\mc H$ and    $W= \overline{\mbox{span}} \{f_k\}$, then  $\{f_k\}_{k\in I}$ is an oblique dual frame of $\{g_k\}_{k\in I}$ on $W$, and $\{g_k\}_{k\in I}$ is an oblique dual frame of $\{f_k\}_{k\in I}$ on $V$ \cite[Lemma 3.1]{christensen2004oblique}. Furthermore, $\{g_k\}_{k\in I}$ and $\{P_V f_k\}_{k\in I}$ are dual frames for $V$ and $\{f_k\}_{k\in I}$ and $\{P_W g_k\}_{k\in I}$ are dual frames for $W$. This decomposition is important to recover data from a given set of samples. Tang in  \cite{Tang2000obliqueprojection} studied the infimum cosine angles in connection with oblique projections that leads to oblique dual frames,   followed by  Kim et al. for the different contexts \cite{kim2001characterizations, kim2003quasi}. Further, Christensen and Eldar in \cite{christensen2004oblique}, and Kim et. al in \cite{kim2005infimum} developed a connection of the infimum cosine angle with oblique dual frames for shift-invariant (SI) spaces in $L^2(\mathbb R^n)$. An existence of Riesz basis using infimum cosine angle for the theory of multiresolution analysis in  $L^2(\mathbb R^n)$ was discussed by  Bownik and Garrig{\'o}s  in \cite{bownik2004biorthogonal}. We aim to continue the work in the context of set-theoretic abstraction.

Now we provide our first main result which is a measure theoretic abstraction of \cite[Theorem 4.10]{kim2005infimum} using range function.  The novelty of considering  the  approach on $L^2 (X; \mc H )$ is  to develop the theory of  duals for a continuous frame on  locally compact group (not necessarily abelian)  translated by its abelian subgroup.

\begin{thm}\label{Result-I}
	Let $(X, \mu_{X})$ and $(\mc M, \mu_{\mc M})$   be   $\sigma$-finite measure spaces such that $\mu (X)<\infty$,  and the set  $\mc D=\{\varphi_s\in L^\infty(X): s\in \mc M\}$ is  Parseval determining set for $L^1(X)$.   
	For the finite collection of functions ${\msc A}=\{f_i\}_{i=1}^m$ and ${\msc B}=\{g_i\}_{i=1}^n$ in $L^2(X; \mc H)$,  and for a.e. $x\in X$, assume the range functions  $J_{\msc A}(x)=\Span\{f_i(x): i=1,2, \dots,m \}$ and $J_{\msc B}(x)=\Span\{g_i(x): i=1,2, \dots,n \}$ associated with the  MI spaces $\mc S_{\mc D}(\msc A)$ and $\mc S_{\mc D}(\msc B)$, respectively. Then the following are equivalent:
	\begin{enumerate}
		\item[(i)]  	 There exist  ${\msc A'}=\{f_i'\}_{i=1}^r$ and ${\msc B'}=\{g_i'\}_{i=1}^r$ in $L^2(X; \mc H)$ such that  $\mc E_{\mc D}(\msc A')$ and  $\mc E_{\mc D}(\msc B')$ are continuous frames for $\mc S_{\mc D}(\msc A)$  and $\mc S_{\mc D}(\msc B)$, respectively, satisfying 
		the  following reproducing formulas for  $g\in \mc S_{\mc D}(\msc A)$ and $h \in \mc S_{\mc D}(\msc B)$:
		\begin{align}\label{GlobalOblique}
			g=\sum_{i=1}^r \int_{\mc M} \langle g, M_{\phi_s}g_i'\rangle M_{\phi_s}f_i' \  \ds,\,  \mbox{and}\   h=\sum_{i=1}^r  \int_{\mc M}\langle h, M_{\phi_s}f_i'\rangle M_{\phi_s}g_i'  \ \ds. 
		\end{align}
		\item[(ii)] The infimum cosine angles  of  $\mc S_{\mc D}(\msc A)$ and $\mc S_{\mc D}(\msc B)$ are greater than zero, i.e.,  $$
		R(\mc S_{\mc D}(\msc A),\mc S_{\mc D}(\msc B))>0 \ 
		\mbox{and} \  R(\mc S_{\mc D}(\msc B),\mc S_{\mc D}(\msc A))>0.$$
		\item[(iii)]  There exist collection of functions  $\{f_i'\}_{i=1}^r$ and $\{g_i'\}_{i=1}^r$ in $L^2(X; \mc H)$ such that for a.e. $x\in X$, the systems $\{f_i'(x)\}_{i=1}^r$ and  $\{g_i'(x)\}_{i=1}^r$ are finite frames for $J_{\msc A}(x)$  and $J_{\msc B}(x)$, respectively, satisfying 
		the  following reproducing formulas  for   $u\in J_{\msc A}(x)$ and   $v\in J_{\msc B}(x)$:  
		\begin{align}\label{LocalOblique}
			u=\sum_{i=1}^r \langle u, g_i'(x)\rangle f_i'(x), \ \mbox{and} \ v=\sum_{i=1}^r \langle v, f_i'(x)\rangle g_i'(x), \  a.e. \ x\in X.
		\end{align}
		
		\item[(iv)]  For a.e. $x\in X$, the infimum cosine angles of  $J_{\msc A}(x)$ and 
		$J_{\msc B}(x)$  are greater than zero, i.e.,
		$$R(J_{\msc A}(x),J_{\msc B}(x))>0 \ 
		\mbox{and} \  R(J_{\msc B}(x),J_{\msc A}(x))>0.$$
	\end{enumerate}
\end{thm}

The equations (\ref{GlobalOblique}) and (\ref{LocalOblique}) explore the various possibilities of obtaining oblique dual frames in the global and local setups, respectively. These duals and associated reproducing formulas are not necessarily unique.  But when $\mc E_{\mc D}(\msc A)$ is a Riesz basis then the dual is always unique. 
 The following main result discusses the uniqueness of reproducing formula, which is a measure-theoretic abstraction of \cite[Corrollary 2.4]{Tang2000obliqueprojection} and \cite[Proposition 2.13]{bownik2004biorthogonal}. 
\begin{thm}\label{Result-II} 
	
	Let $(X, \mu_{X})$   be   a $\sigma$-finite measure space with  $\mu (X)<\infty$, and    let  $\msc V$ and $\msc W$ be   multiplication invariant subspaces of $L^2(X;\mc H)$ corresponding to an orthonormal basis   $\msc D$ of $L^2(X)$. For the finite collection of functions ${\msc A}=\{f_i\}_{i=1}^r$,  assume $\mc E_{\msc D} (\msc A)$ is a Riesz basis for $\msc V$.
%	 and $J_{\msc A}(x)=\Span\{f_i(x)\}_{i=1}^r$ is its associated range function for a.e. $x\in X$ 
%	and it satiesfies the following dimension condition:
%	$$\dim J_{\msc V}(x)=\dim J_{\msc W}(x)=r \ \mbox{for a.e.} \ x \in \ X.$$
	Then the following holds:
	\begin{enumerate}
		\item[(i)] \textbf{Global setup:} If  there exists  ${\msc A'}=\{f'_i\}_{i=1}^r$  in $L^2(X; \mc H)$ such that $\mc E_{\msc D} (\msc A')$ is a Riesz basis for $\msc W$ satisfying the  following biorthogonality  condition 
		\begin{align}\label{biorthogonal}
			\langle M_{\phi} f_i, M_{\phi'}f'_{i'} \rangle=\delta_{i, i'}\delta_{\phi, \phi'}, \quad i, i'=1, 2, \cdots, r; \  \phi, \phi'\in \msc D, 
		\end{align}
		then 				the infimum cosine angles  of $\msc V$ and $\msc W$ are greater than zero, i.e.,  
		\begin{align}\label{Infimum}
			R(\msc V,\msc W)>0 \ 
			\mbox{and} \  R(\msc W, \msc V)>0. 
		\end{align}
		Conversely	if  (\ref{Infimum})  holds true, then  there exists   ${\msc A'}=\{f'_i\}_{i=1}^r$  in $L^2(X; \mc H)$ such that $\mc E_{\msc D} (\msc A')$ is a Riesz basis for $\msc W$ satisfying the  biorthogonality  condition (\ref{biorthogonal}).  
		Moreover, the following reproducing formulas hold: 
		$$f=\sum_{\phi \in \msc D} \sum_{i=1}^r\langle f, M_{\phi}f_i'\rangle M_{\phi}f_i,  \  \forall f\in \msc V,  \   \mbox{and} \  g=\sum_{\phi\in \msc D} \sum_{i=1}^r \langle g, M_{\phi}f_i\rangle M_{\phi}f_i',  \  \forall g\in \msc W.$$
		
		\item[(ii)]  \textbf{Local setup:} If  there exists  ${\msc A'}=\{f'_i\}_{i=1}^r$  in $L^2(X; \mc H)$  such that for a.e. $x \in X$, $\{f'_i(x)\}_{i=1}^r$  is a  Riesz sequence in $\mc H$   satisfying the  following biorthogonality  condition 
		\begin{align}\label{Localbiorthogonal}
			\langle   f_i(x),  f'_{i'}(x) \rangle=\delta_{i, i'}, \quad i, i'=1, 2, \cdots, m,    \ a.e. \  x \in X, 
		\end{align}
		the infimum cosine angles  of $J_{\msc A}(x)=\Span \{f_i(x)\}_{i=1}^r$ and $J_{\msc A'}(x)=\Span \{f_i'(x)\}_{i=1}^r$ are greater than zero, i.e.,  
		\begin{align}\label{LocalInfimum}
			R(J_{\msc A}(x), J_{\msc A'}(x))>0 \ 
			\mbox{and} \  R(J_{\msc A'}(x), J_{\msc A}(x))>0, \ a.e. \ x \in X. 
		\end{align}
		Conversely if (\ref{LocalInfimum}) holds,      there exists   ${\msc A'}=\{f'_i\}_{i=1}^r$  in $L^2(X; \mc H)$ such that  for  a.e. $x\in X$,  $\{f_i'(x)\}_{i=1}^r$ is a Riesz sequence in  $\mc H$ satisfying the  biorthogonality  condition (\ref{Localbiorthogonal}).  Moreover,  the following reproducing formulas hold for $u\in J_{\msc A}(x)$,   and $v\in J_{\msc A'}(x)$:  
		\begin{align*} 
			u=\sum_{i=1}^r \langle u, f_i'(x)\rangle f_i(x), \ \mbox{and} \ v=\sum_{i=1}^r \langle v, f_i(x)\rangle f_i'(x), \ \mbox{for a.e.} \ x \in X.
		\end{align*}
	\end{enumerate}

\end{thm}

The paper is organized as follows; in Section \ref{Oblique}, we have discussed multiplication-generated oblique dual frames and their characterizations in connection with the Gramian matrix. Then in Section \ref{S: Proof}  the proofs of the Theorem \ref{Result-I}  and \ref{Result-II} are provided. The paper ends with Section \ref{Application} which discusses the applications to the locally compact group translated by the closed abelian subgroup.
%\section{Preliminary}\label{S:Pre}
%\begin{lem}
%If $P: P_V|_U$ is one to one and range $P$ is invarinat under $S_V$, then  
%$$G^\dagger= T_U^\dagger P^\dagger (T_V^*)^\dagger= T_U^\dagger P^\dagger (T_V^\dagger)^*$$
%\end{lem}
%\begin{lem}
%Suppose that $U$ is not trivial. Then
%\begin{align*}
%R(U, V)=\begin{cases}
%0, \mbox{If P is one-to-one}\\
%\|P^\dagger\|^{-1}
%\end{cases}
%\end{align*}
%\end{lem}
%\begin{lem}
%Suppose one of the following two condition holds:
%\begin{enumerate}
%\item[(i)] $0<rank G=dim U = dim V$\\
%\item[(ii)] $0<rank G=dim U < dim V$ and ran P is invariant under $S_V$.
%\end{enumerate}
%Then $R(U, V)= \|G_U^\frac{1}{2}G^\dagger G_V^\frac{1}{2}\|^{-1}$.
%\end{lem}
%\begin{lem}
%Suppose that $R(U, V)>0$, then the following assertion are equivalent 
%\begin{enumerate}
%\item[(i)] $R(U, V)=R(V, U)$
%\item[(ii)] $R(V, U)>0$
%\item[(iii)] $\dim U=\dim V$
%\item[(iv)] $\rank G_U= \rank G_V$.
%\end{enumerate}
%\end{lem}

\section{Multiplication generated oblique dual frames}
\label{Oblique}
Given a Parseval determining set $\mc D :=\{g_s\in L^\infty(X):s\in \mc M\}$ for $L^1(X)$ (see, (\ref{Parsevaldeter})), and a finite collection  of functions
$\msc A =\{\varphi_i\}_{i \in \mc I_r}$  in  $L^2(X; \mathcal H)$,   we recall  the MG system $\mc E_{\mc D}(\msc A)$  and its   associated MI   space  $\mc S_{\mc D}(\msc A)$  
given by
\begin{align}\label{ED}
 \mc E_{\mc D}(\msc A) :=\left\{M_{g_s}\varphi_i (\cdot)=g_s(\cdot)\varphi_i (\cdot): s\in \mc M, i \in \mc I_r \right\},  \quad \mbox{and} \quad   \mc S_{\mc D}(\msc A) :=\ol{\Span} \, \mc E_{\mc D}(\msc A),
\end{align}
respectively, where 
$(\mc M, {\mu_{\mc M}})$ is a $\sigma$-finite measure space,  and $\mc I_r :=\{1, 2, \cdots, r\}$, for $r \in \mathbb N$. 
%   Note that   
%$E_{\mc D}(\msc A)$ is  a   continuous frame (Bessel) for the MI  space
%$S_{\mc D}(\msc A)$  with bounds $A$ and $B$    if and only if   the system $\msc A (x)$ is  a  continuous frame (Bessel) for   $J_\msc A (x)$ with bounds $A$ and $B$, for a.e. $x \in X$.   
The MG system $\mc E_{\mc D}(\msc A)$  
is said to be  a \textit{continuous frame}  (simply, \textit{frame}) for  $\mc S_{\mc D}(\msc A) $ if  the map    $(s, i) \mapsto  \langle f, M_{g_s}\varphi_i\rangle$ from $(\mc M\times \mc I_r)$ to $\mathbb C$ is measurable,  and there exist $0<A\leq B<\infty$ such that 
\begin{equation}\label{framedefn}
	A\|f\|^2\leq \sum_{i \in \mc I_r} \int_{\mc M} | \langle f, M_{g_s} \varphi_i\rangle |^2 d_{\mu_{\mc M}} (s)   \leq B\|f\|^2,\ \text{for all}
	\  f\in S_{\mc D}(\msc A) .
\end{equation}
If $S_{\mc D}(\msc A) =L^2(X;\mc H)$,           $\mc E_{\mc D}(\msc A)$ is  a  \textit{frame} for $L^2(X;\mc H)$, and it is \textit{Bessel} in $L^2(X;\mc H)$   when  only upper bound holds in (\ref{framedefn}).  
% For two collection of functions $\msc A$ and $\msc A'$,  suppose the systems $ E_{\mc D}(\msc A)$  and $E_{\mc D}(\msc A')$  are Bessel in $L^2(X;\mc H)$ and  they  satisfy the following reproducing formula: 
%\begin{equation}\label{subspace dual}
%	\int_{\mc N} \int_{\mc M} \langle f, g_s\psi_t\rangle g_s\varphi_t d_{\mu_{\mc M}} (s)   d_{\mu_{\mc N}} (t) = f, \  \text{for all}
%	\  f\in S_{\mc D}(\msc A), 
%\end{equation}	
%we 	call $E_{\mc D}(\msc A')$ is a alternate dual  to $E_{\mc D}(\msc A)$.
%Every frame or a Bessel family is associated with an analysis operator, the range of which carries out a lot of information of a signal/image or function. 

For a Bessel family $ \mc E_{\mc D}(\msc A)$ in $L^2(X; \mc H)$,   we define a bounded linear operator $T_{ \mc E_{\mc D}(\msc A)} :L^2(X;\mc H)\ra L^2(\mc M\times \mc I_r)$,  known as \textit{analysis operator},   by
\begin{equation*}\label{AnalysisOp}
	T_{ \mc E_{\mc D}(\msc A)}  (f)(s, i)=\langle f, M_{g_s}\varphi_i\rangle, \   \mbox{for all}   \ (s,i)\in \mc M\times \mc I_r, \, \mbox{and} \, f \in L^2(X;\mc H), 
\end{equation*}
and  its adjoint operator $T_{\mc E_{\mc D}(\msc A)}^*: L^2(\mc M\times \mc I_r)\ra L^2(X; \mc H)$, known as \textit{synthesis operator}, by
\begin{equation*}\label{SynthesisOp}
	T_{\mc E_{\mc D}(\msc A)}^* \psi  =  \sum_{i\in \mc I_r}\int_{\mc M} \psi(s, i)  M_{g_s}\varphi_i  \  d_{\mu_{\mc M}} (s),\   \mbox{for all}   \  \psi \in L^2(\mc M\times \mc I_r),
\end{equation*}
in the    weak sense.   Then,   the composition    $S_{\mc E_{\mc D}(\msc A)} :=T_{\mc E_{\mc D}(\msc A)}^*T_{\mc E_{\mc D}(\msc A)}:L^2(X;\mc H) \ra L^2(X;\mc H)$  is known as  \textit{frame operator}.

  At this juncture it can be noted that the  Bessel family ${\mc E_{\mc D}(\msc A)}$ is  a  continuous frame for ${\mc S_{\mc D}(\msc A)} $ with bounds $0<A \leq B$ if and only if  
 the frame operator $S_{\mc E_{\mc D}(\msc A)}\big|_{{\mc S_{\mc D}(\msc A)}}$ restricted on ${\mc S_{\mc D}(\msc A)}$ is positive, bounded and invertible with  $ A I_{{\mc S_{\mc D}(\msc A)}} \leq S_{\mc E_{\mc D}(\msc A)}\big|_{{\mc S_{\mc D}(\msc A)}}\leq B I_{{\mc S_{\mc D}(\msc A)}}$, where $I_{{\mc S_{\mc D}(\msc A)}}$ denotes the identity operator on $L^2(X;\mc H)$ which is restricted on     ${\mc S_{\mc D}(\msc A)}$. The inverse of frame operator $S_{\mc E_{\mc D}(\msc A)}$ satisfies $ \frac{1}{B} I_{{\mc S_{\mc D}(\msc A)}} \leq (S_{\mc E_{\mc D}(\msc A)}\big|_{{\mc S_{\mc D}(\msc A)}})^{-1}\leq \frac{1}{A} I_{{\mc S_{\mc D}(\msc A)}}$ and the family $\{  (S_{\mc E_{\mc D}(\msc A)}\big|_{{\mc S_{\mc D}(\msc A)}})^{-1} S_{\mc D}(\msc A)\}$
  is also a  continuous frame for ${\mc S_{\mc D}(\msc A)}$,  known as \textit{canonical dual frame} of ${\mc E_{\mc D}(\msc A)}$,  which satisfies the following reproducing formula for all $f \in{\mc S_{\mc D}(\msc A)}$ in the weak sense:  
  \begin{equation}\label{cont.canonical}
 f=\sum_{i\in \mc I_r}\int_{\mc M}\lan f,  (S_{\mc E_{\mc D}(\msc A)}\big|_{{\mc S_{\mc D}(\msc A)}})^{-1} M_{g_s}\varphi_i\ran M_{g_s}\varphi_i \, d_{\mu_{\mc M}}(s).
 \end{equation}
 The reproducing formula (\ref{cont.canonical})  gives an idea to find a new Bessel family, say $\{h_i\}_{i\in \mc I_r}=:\msc A'$ in  $L^2(X;\mc H)$, such that  the following decomposition  formula holds:   
 $$  
  f=\sum_{i\in \mc I_r}\int_{\mc M}\lan f, M_{g_s}h_i\ran M_{g_s}\varphi_i \, d_{\mu_{\mc M}}(s),   \ f \in {\mc S_{\mc D}(\msc A)}, \quad   i.e., \quad  T_{\mc E_{\mc D}(\msc A)}^*T_{\mc E_{\mc D}(\msc A')}\big|_{{\mc  S_{\mc D}(\msc A)}}=I_{\mc  S_{\mc D}(\msc A)},
 $$
 where $\mc  S_{\mc D}(\msc A')$ need not be a subset of ${\mc S_{\mc D}(\msc A)}$. It motivates to define duals other than canonical dual. 
 
 Next we define alternate and oblique duals for ${\mc E_{\mc D}(\msc A)}$. Askari and Gabardo  in \cite{gabardo2007uniqueness} introduced such duals for shift invariant subspaces of $L^2(\mathbb R^n)$ and   Heil et al.   in \cite{heil2009duals} defined them for a separable Hilbert space.

\begin{defn}\label{MG-dual}  Let  
	$\mc E_{\mc D}(\msc A)$ be a continuous frame for  
	$\mc S_{\mc D}(\msc A)$ and  $\mc E_{\mc D}(\msc A')$ be a Bessel family in  $L^2(X; \mathcal H)$. Then
\begin{enumerate}
\item[(i)]   $\mc E_{\mc D}(\msc A')$  is an \textit{alternate  MG-dual}  for   $\mc E_{\mc D}(\msc A)$  if $T_{\mc E_{\mc D}(\msc A)}^*T_{\mc E_{\mc D}(\msc A')}\big|_{\mc S_{\mc D}(\msc A)} = I_{\mc S_{\mc D}(\msc A)}$. 
\item[(ii)]  The    alternate  MG-dual $\mc E_{\mc D}(\msc A')$   is called an \textit{oblique  MG-dual}  for   $\mc E_{\mc D}(\msc A)$ if  $\mc E_{\mc D}(\msc A')$ is a continuous frame for  
$\mc S_{\mc D}(\msc A')$ and $T_{\mc E_{\mc D}(\msc A')}^*T_{\mc E_{\mc D}(\msc A)}\big|_{\mc S_{\mc D}(\msc A')} = I_{\mc S_{\mc D}(\msc A')}$. 
\item[(iii)]  The   oblique  MG-dual $\mc E_{\mc D}(\msc A')$  is an  \textit{MG-dual frame}  for   $\mc E_{\mc D}(\msc A)$ if  $\mc S_{\mc D}(\msc A')=\mc S_{\mc D}(\msc A)$. 
\end{enumerate}
\end{defn} 
 
In this section, we aim to characterize alternate and oblique duals for ${\mc E_{\mc D}(\msc A)}$. For this, we need the concept of Fourier transform in the abstract setup. We now provide a notion of Fourier transform for $L^2(X)$, introduced by Bownik and  Iverson in \cite{BownikIversion2021}.

\begin{defn}\label{ft}  For $f \in L^1(X) \bigcap L^2(X)$, the \textit{Fourier transform} $\mathcal F f \in L^2(\mathcal M)$ corresponding to the Parseval determining set $\mathcal D=\{\phi_s\in L^{\infty}(X): s \in \mc M\}$ is given  by 
	\begin{align}\label{FT}
	(\mc Ff)(s)=\int_{X}f(x)\ol{g_s (x)}d_{\mu_{X}}(x), \qquad  ~ a.e.~ s\in \mc M.
	\end{align}
	The  Fourier transform $\mathcal F$   is a unique extension from $L^2(X)$ to $L^2(\mathcal M)$ which is  linear and  isometry. 
\end{defn}
The \textit{Plancherel's relation} and \textit{Parseval's formula}  are given by  
\begin{align}\label{plancherel}
\|\mc Ff\|_{L^2(\mathcal M)}=\|f\|_{L^2(X)} \   \mbox{and} \   \lan \mc Ff, \mc Fh\ran_{L^2(\mc M)}=\lan f, h\ran_{L^2(X)} , \ \mbox{for all} \ f, h \in L^2 (X),
\end{align} 
respectively. 

The following result gives a way to move from global setup to local setup. For $\msc B \subset L^2(X; \mc H)$ and $x \in X$, the set $\msc B(x)$ is given by $
\msc B(x) := \{f(x): f \in \msc B\},$ which will be frequently used in the sequel.   

\begin{prop} \label{BesselCharrecterisation} Let $\msc A$ and $\msc A'$ be   finite collections of functions  in $L^2(X; \mc H)$ having same cardinality such that  $\mc E_{\mc D}(\msc A)$ and  $\mc E_{\mc D}(\msc A')$ are Bessel.  The  following holds true for all $f, g \in L^2(X; \mc H)$: 
	$$ 
	\left \langle T_{\mc E_{\mc D}(\msc A)}^*  T_{\mc E_\mc D(\msc A')}f,   g\right \rangle=\int_{X} \langle    T_{\msc A (x)} ^*   T_{\msc A' (x)} f(x),      g(x) \rangle  \domega, 
	$$
	where   the operators $T_{\msc A(x)}$ and  $T_{\msc A'(x)}$ are  analysis operators associated to    $\msc A(x)$ and $\msc A'(x)$, respectively, for a.e. $x \in X$.
%	%Moreover, we obtain 
%	$$
%	\sum_{i =1}^r	\displaystyle \int_{\mc M}  \lan f, M_{\phi_s} g_i\ran \overline{\lan g, M_{\phi_s} f_i\ran}     d_{\mu_{\mc M}}(s)=\sum_{i=1}^r \int_{X}  \lan f(x), \psi_i (x)\ran   \overline{\lan g(x),\varphi_i (x)\ran}         \dx.
%	$$
\end{prop}

\begin{proof} Let $\msc A=\{f_i\}_{i=1}^r$ and $\msc A'=\{g_i\}_{i=1}^r$ be two finite collections of functions  in $L^2(X; \mc H)$.  For  $f, g \in L^2(X; \mc H)$, the analysis operators $T_{\mc E_\mc D(\msc A)}$  and $T_{\mc E_\mc D(\msc A')}$ satisfy   $$(T_{\mc E_\mc D(\msc A)}g) (s,i)=\lan g, M_{\phi_s} \varphi_i\ran \ \mbox{and} \ (T_{\mc E_\mc D(\msc A')}f) (s,i)=\lan f, M_{\phi_s} \psi_i\ran, \  \mbox{for all} \ (s,i) \in \mc M \times \{1,\dots, r\},  $$ 
	and then  we compute the following:	
	{\small	
		\begin{align*}
		\left \langle T_{\mc E_\mc D(\msc A')}f,   T_{\mc E_{\mc D}(\msc A)} g\right \rangle 
		= & \int_{\mc M}  \sum_{i=1}^r\lan f, M_{\phi_s} \psi_i\ran  \overline{\lan g, M_{\phi_s} \varphi_i \ran} \,  d_{\mu_{\mc M}}(s) \\
		= & \int_\mc M  \sum_{i=1}^r \left(\int_X\langle f(x), \phi_s(x)\psi_i(x)\rangle  \domega\right)\times \overline{\left(\int_{ X}\langle g(x), \phi_s(x)\varphi_i(x) \rangle \domega\right)}  \dm\\
		= & \int_\mc M  \sum_{i=1}^r \left(\int_{X}\langle f(x),\psi_i(x)\rangle \overline{\phi_s(x)} \domega \right)\times \overline{\left(\int_{X}\langle  g(x), \varphi_i(x)\rangle  \overline{\phi_s(x)} \domega\right)}  \dm.
		\end{align*}
	}
	Choosing  $F_{\psi_i}(x)=\lan f(x),\psi_i (x)\ran$ and $G_{\varphi_i}(x)=\lan g(x),\varphi_i (x)\ran$, for $x \in X$ and $i \in \{1,\dots, r\}$, we have 		
	\begin{align*}
	\left \langle T_{\mc E_\mc D(\msc A')}f,   T_{\mc E_{\mc D}(\msc A)} g\right \rangle 
	= & \int_\mc M   \sum_{i=1}^r \mc F F_{\psi_i} (s)  \overline{\mc F G_{\varphi_i} (s)}    \dm =  \sum_{i=1}^r  \int_\mc M  \mc F F_{\psi_i} (s)  \overline{\mc F G_{\varphi_i} (s)}    \dm \\	
	=&   \sum_{i=1}^r\langle   \mc F F_{\psi_i},   \mc F G_{\varphi_i} \rangle     
	=  \sum_{i=1}^r \langle     F_{\psi_i},     G_{\varphi_i} \rangle      \\
	= &  \sum_{i=1}^r \int_X   F_{\psi_i} (x) \overline{G_{\varphi_i}(x)}   \dx     \\
	=&  \sum_{i=1}^r \int_X
	\lan f(x),\psi_i (x)\ran \overline{\lan g(x),\varphi_i (x)\ran}  \dx,
	\end{align*}
	using Fourier transform in (\ref{FT}),  Parseval's formula (\ref{plancherel}) on $L^2(X)$ and Fubini's theorem over $\mc M \times \{1,2,\dots, r\}$, $F_{\psi_i},  G_{\varphi_i} \in   L^2 (X; \mc H)$ hold in the above calculations  by noting the facts that $\mc E_\mc D(\msc A)$ is Bessel systems with bound $B$ if and only if $\msc A(x)$  is Bessel systems with bound $B$ for a.e. $x \in X$, and the following estimate
	\begin{align*}
	\int_X   \sum_{i=1}^r   \left| F_{\psi_i} (x) \overline{G_{\varphi_i}(x)} \right| \dx    
	\leq &\left(\int_X    \sum_{i=1}^r\left| \langle f(x), \psi_{t}(x)\rangle\right|^2 \dx \right)^\frac{1}{2}
	\times\left(\int_X    \sum_{i=1}^r \left|\langle g(x), \varphi_{t}(x)\rangle\right|^2 \dx\right)^\frac{1}{2} \\
	\leq &   \sqrt{B B'}  \|f\| \|g\|,
	\end{align*}
	using Cauchy-Schwarz inequality, where we assume $\mc E_\mc D(\msc A)$ and $\mc E_\mc D(\msc A')$ are Bessel systems with bounds $B$ and $B'$, respectively. 	Therefore using Fubini's theorem over   $\mc N \times X$, we get
	\begin{align*}
	\left \langle T_{\mc E_\mc D(\msc A')}f,  T_{\mc E_{\mc D}(\msc A)} g\right \rangle 
	= &  \sum_{i=1}^r \int_X      \lan f(x),\psi_i (x)\ran \overline{\lan g(x),\varphi_i (x)\ran}  \dx    \\
	= &      \int_X   \sum_{i=1}^r T_{\msc A' (x)}   (f (x))(i) \overline{T_{\msc A (x)}   (g(x))(i)}   \dx \\
	=  &   \int_X  \langle      T_{\msc A' (x)}   f(x),    T_{\msc A (x)} g(x) \rangle  \dx, 
	\end{align*}
	where 	$T_{\msc A' (x)}  (f (x))(i)=\lan f(x),\psi_i (x)\ran$ and $T_{\msc A (x)}   (g(x))(i)=\lan g(x),\varphi_i (x)\ran$, for $i=1,2,\dots, r$.  	  
\end{proof}

Next we provide a characterization of alternate dual  associated with the Gramian operators. The \textit{Gramian} and \textit{dual Gramian operators} are given 
$$
G_{\msc A}(x)=T_{\msc A}(x)T_{\msc A}^*(x) \ \mbox{and} \  \tilde{G}_{\msc A}(x)=T_{\msc A}^*(x)T_{\msc A}(x), \ a.e. \ x\in X,   
$$
where $T_{\msc A}(x)$ and $T_{\msc A}^*(x)$ denote the  analysis and synthesis operators corresponding to $\msc A(x)=\{\varphi_i(x)\}_{i \in \mc I_r}$.  For   $\msc A=\{\varphi_i\}_{i\in \mc I_r}$ and $\msc A'=\{\psi_i\}_{i\in \mc I_r}$,  the operator  $G_{\msc A, \msc A'}(x)=	\left[\langle\varphi_j(x),\psi_i(x)\rangle\right]_{i, j \in \mc I_r}$ is known as   the \textit{mixed  Gramian operator}, for a.e. $x\in X$.  The following result is a measure theoretic abstraction of \cite[Theorem 4.1]{kim2005infimum} and \cite[Theorem 5(a)]{gabardo2007uniqueness}.

\begin{prop}\label{AlternateDual}   In addition to  the assumptions of Proposition \ref{BesselCharrecterisation},  let us assume   $\mc E_{\mc D}(\msc A)$ be  a frame for $\mc S_{\mc D}(\msc A)$. Then the  system  $\mc E_{\mc D}(\msc A')$ is  an alternate MG-dual for $\mc E_{\mc D}(\msc A)$ if and only if for a.e.   the system $\msc A'(x)=\{\psi(x): \psi\in \msc A'\}$ is an alternate dual for $\msc A(x)=\{\varphi(x): \varphi\in \msc A\}$, equivalently,
	the Gramian $G_{\msc A}(x)$ and mixed Gramian $G_{\msc A, \msc A'}(x)$  operators satisfy the following relation: 
	$$G_{\msc A}(x)G_{\msc A, \msc A'}(x)=G_{\msc A}(x),  \  \mbox{for a.e.} \  x \in X.$$
\end{prop}

\begin{proof} Suppose the  system  $\mc E_{\mc D}(\msc A')$ is  an alternate MG-dual for $\mc E_{\mc D}(\msc A)$, we have $T^*_{\mc E_{\mc D}(\msc A)} T_{\mc E_{\mc D}(\msc A')}\big|_{S_{\mc D}(\msc A)}=I_{S_{\mc D}(\msc A)}$. 
By Proposition \ref{BesselCharrecterisation}, we get 	\begin{align}\label{AlternateEqn}
		\int_X    \lan T^*_{\msc A (x)} T_{\msc A'(x)}f(x),  g(x)\ran       \dx
		=\int_X\lan f(x), g(x)\rangle \dx, \ \forall f,g \in \sd.
	\end{align}
	At first we will show,   $T^*_{\msc A (x)} T_{\msc A'(x)}\big|_{J_\msc A (x)}=I_{J_\msc A (x)}$,  for a.e. $x \in X$.  For this, let $\{x_n\}_{n \in \mathbb N}$ be  a countable dense subset of $\mc H$ and let $P_{J_\msc A} (x)$ be  an orthogonal projection onto $J_\msc A (x)$ for a.e. $x \in X$. Clearly for a.e. $x \in X$, $\{P_{J_\msc A}  (x) x_n\}_{n \in \mathbb N}$ is    dense in $J_\msc A (x)$.
	Next for each $m,n \in \mathbb N$, we  define a set $S_{m, n}$ as follows:
	\begin{align*}
		S_{m, n}= \Big\{x \in X: \rho_{m, n}(x) := & \lan T^*_{\msc A (x)} T_{\msc A'(x)}   P_{J_\msc A}(x)x_m, P_{J_\msc A}(x)x_n \ran \\
		&   -   \lan P_{J_\msc A}(x)x_m, P_{J_\msc A}(x)x_n\ran\neq \{0\} \Big\}.
	\end{align*}
	Now we assume  on the contrary $T^*_{\msc A (x)} T_{\msc A'(x)}\big|_{J_\msc A (x)} \neq I_{J_\msc A (x)}$  on a Borel measurable subset $Y$ of $X$ having positive measure. Then, there are $m_0, n_0 \in \mathbb N$ such that  $S_{m_0, n_0} \bigcap Y$ is a Borel measurable subset  of $X$ having positive measure, and hence   either real or imaginary parts of $\rho_{m_0, n_0}(x)$ are strictly positive or negative on a.e. $x \in S_{m_0, n_0} \bigcap Y$. Firstly we assume that the real part of $\rho_{m_0, n_0}(x)$ is strictly positive on $S_{m_0, n_0} \bigcap Y$. By choosing   a Borel measurable subset $S$ of $S_{m_0, n_0} \bigcap Y$ having positive measure, we define functions $h_1$ and $h_2$ as follows:  
	%\vspace{-0.5cm}
%	\begin{multicols}{2}
%		\begin{equation*}\label{constructionf}
%			h_1(x)=\begin{cases} 
%				P_{J_\msc A} (x) x_{m_0} & \mbox{for } x \in S,\\
%				0 & \mbox{for } x \in X \backslash S,
%			\end{cases}	 	 
%			\qquad  	\mbox{and} \columnbreak \qquad 
%			h_2(x)=\begin{cases} 
%				P_{J_\msc A} (x) x_{n_0} & \mbox{for } x \in S,\\
%				0 & \mbox{for } x \in X \backslash S.
%			\end{cases}	 
%		\end{equation*}
%	\end{multicols}
	\noindent Then, we have $h_1 (x ), h_2 (x )\in J_\msc A (x)$,  for a.e. $x \in X$ since $\{P_{J_\msc A}  (x) x_n\}_{n \in \mathbb N}$ is    dense in $J_\msc A (x)$, and hence we get $h_1, h_2 \in \sd$ in view of \cite[Theorem 2.4]{bownik2015structure}. Therefore using $f=h_1$,  $g=h_2$ in  (\ref{AlternateEqn}), we obtain  	 	$\int_S   \rho_{m_0, n_0}(x )        \dx =0$ which is a contradiction since the measure of $S$ is positive and the real part of $\rho_{m_0, n_0}(x)$ is strictly positive on $S$. Other cases follow in a similar way.  Since $J_{\msc A}(x)={\mbox{span}}\{\varphi_i(x)\}_{i=1}^r$, for each $i=1,2\dots, r$, we have,
	$\varphi_i(x)=\sum_{i=1}^{r}\langle \varphi_i(x), \psi_i(x)\rangle \varphi_i(x)$, which is equivalent to 
	$$T^*_{\msc A (x)} T_{\msc A' (x)}\big|_{J_\msc A (x)} = I_{J_{\msc A (x)}},  \mbox{for a.e.}\  x \in X.$$
	 Equivalently, we have $G_{\msc A}(x)G_{\msc A, \msc A'}(x)=G_{\msc A}(x),  \  \mbox{for a.e.} \  x \in X, $ by looking at the definition of Gramian and mixed Gramian operators. Therefore we get the result.

	Conversely,  assume  $G_{\msc A}(x)G_{\msc A, \msc A'}(x)=G_{\msc A}(x),  \  \mbox{for a.e.} \  x \in X$, equivalently,   $T^*_{\msc A (x)} T_{\msc A' (x)}\big|_{J_\msc A (x)} = I_{J_{\msc A (x)}}$,  for a.e. $x \in X$. Then  we have     $T^*_{\mc E_{\mc D}(\msc A)} T_{\mc E_{\mc D}(\msc A')}\big|_{S_{\mc D}(\msc A)}=I_{S_{\mc D}(\msc A)}$,    follows from  the computation 
 	\begin{align*}
		\left \langle  T^*_{\mc E_{\mc D}(\msc A)} T_{\mc E_\mc D(\msc A')}f,  g\right \rangle 
		&=\int_X \left \langle   T_{ \msc A'(x)}f(x),  T_{\msc A(x)}g(x)\right \rangle \dx\\  
		&=\int_X \left \langle   T^*_{\msc A(x)}T_{ \msc A'(x)}f(x),  g(x)\right \rangle \dx\\  
		&=\int_X \left \langle    f(x),  g(x)\right \rangle \dx, \ \forall  f, g \in S_{\mc D}(\msc A)
	\end{align*}
  in view of  Proposition \ref{BesselCharrecterisation}.  
  	\end{proof}

% Let $T_{\msc A}(x)$ and $T_{\msc A}^*(x)$ be the associated  analysis and synthesis operator corresponding to the  collection $\msc A(x)=\{\varphi_i(x):i=1,2,\dots, r\}$,  for a.e. $x\in X$. The \textit{Gramian} and \textit{dual Gramian operators} are defined by 
%$$
%G_{\msc A}(x)=T_{\msc A}(x)T_{\msc A}^*(x) \ \mbox{and} \  G_{\msc A}(x)=T_{\msc A}^*(x)T_{\msc A}(x), \ a.e. \ x\in X.
%$$
%In terms of matrix representation,  we have: 
%{\small 
%	\begin{align}\label{Gramian}
%		G_{\msc A}(x)=\left[\langle\varphi_i(x),\varphi_j(x)\rangle\right]_{i, j \in \mc I_r} \ \mbox{and} \ \tilde{G}_{\msc A}(x)=\left[\, \sum_{i \in \mc I_r} \langle e_k (x), \varphi_i(x) 	\rangle \overline{\langle e_l (x), \varphi_i(x) 	\rangle} \, \right]_{k, l \in \mathcal J}, 
%\end{align} }
%where   $\{e_k (x)\}_{k \in \mc J}$ ($\mc J$-countable index set)  is  the standard orthonormal basis  for $\mc H$.
%For   $\msc A=\{\varphi_i\}_{i\in \mc I_r}$ and $\msc A'=\{\psi_i\}_{i\in \mc I_r}$,  the operator  $G_{\msc A, \msc A'}(x)=	\left[\langle\varphi_j(x),\psi_i(x)\rangle\right]_{i, j \in \mc I_r}$ is known as   the \textit{mixed  Gramian operator}, for a.e. $x\in X$.
 
  The following result is a measure-theoretic abstraction of \cite[Theorem 4.1]{kim2005infimum} for oblique dual frames associated with the rank of mixed Gramian operator and the dimension of range functions. 

\begin{prop}\label{Oblique-Dual}  In addition to  the assumptions of Proposition \ref{BesselCharrecterisation}, let us assume $\mc E_{\mc D}(\msc A)$ and  $\mc E_{\mc D}(\msc A')$  be   frames for $\mc S_{\mc D}(\msc A)$  and $\mc S_{\mc D}(\msc A')$, respectively, such that $\mc E_{\mc D}(\msc A')$ is an alternate MG-dual for $\mc E_{\mc D}(\msc A)$ and 
	\begin{align}\label{rank}
	\rank \  G_{\msc A, \msc A'}(x)=\dim J_{\msc A} (x)=\dim J_{\msc A'} (x), \  a.e. \ x\in X,
	\end{align}
	where  $J_{\msc A}(x)=\Span\{f(x):  f \in \msc A\}$ and $J_{\msc A'}(x)=\Span\{g(x): g \in \msc A'\}$. 
Then 
$\mc E_{\mc D}(\msc A')$ is an oblique MG-dual for $\mc E_{\mc D}(\msc A)$.
 \end{prop}
\begin{proof}  	Observing the proof of  Proposition \ref{AlternateDual}, we get $\msc A'(x)$ is an alternate dual to $\msc A(x)$, for a.e. $x\in X$ since $\mc E_{\mc D}(\msc A')$ is an alternate MG-dual for $\mc E_{\mc D}(\msc A)$. Then for a.e. $x\in X$, we can write $\varphi(x)=\sum_{i=1}^r\langle \varphi(x), g_i(x)\rangle f_i(x)$, for each $\varphi\in \mc S_{\mc D}(\msc A)$.  Further note that $P(x):= P_{J_{\msc A'(x)}}|_{ J_{\msc A}(x)}: J_{\msc A(x)}\ra J_{\msc A'(x)}$ is invertible    in view of \cite[Lemma 3.1]{kim2005infimum}  and relation (\ref{rank}).  Therefore for $k=1,2, \dots, r$ and a.e. $x\in X$,  we get
	  \begin{align*}
\langle P(x)\varphi(x), g_k(x)\rangle& =\langle P_{J_{\msc A'(x)}}\varphi(x), g _k(x)\rangle=\langle \varphi(x),P_{J_{\msc A'(x)}}g _k(x)\rangle=\langle \varphi(x),g _k(x)\rangle\\
&=\left\langle\langle \sum_{i=1}^r\varphi(x),g _i(x)\rangle f_i(x),g _k(x)\right\rangle\\
 &=
\left\langle \varphi(x), \sum_{i=1}^r \langle g _k(x),f_i(x)\rangle g _i(x)\right\rangle\\
&=
\left\langle P(x)\varphi(x), \sum_{i=1}^r \langle g _k(x),f_i(x)\rangle g _i(x)\right\rangle. 
\end{align*}
 and hence 
$g _k(x)=\sum_{i=1}^r \langle g _k(x),f_i(x)\rangle g _i(x)$ since $P(x)$ is invertible. Hence the result holds by noting Proposition \ref{AlternateDual}. 
\end{proof}

The next  result   tells that the space $L^2(X; \mc H)$ can be decomposed   with the help of $\mc S_{\mc D}(\msc A)$ and $\mc S_{\mc D}(\msc A')$ using the rank condition (\ref{rank}).  We use  the angle between two MI subspaces and their point-wise characterizations for its proof.
%The following  proposition  is a measure theoretic abstraction of \cite[Proposition 2.10]{bownik2004biorthogonal} and \cite[Proposition 4.2]{kim2005infimum}. 
	From Definition \ref{D-InfCosine}, note that  
	$$(P_{\mc S_{\mc D}(\msc A)}|_{\mc S_{\mc D}(\msc A')}f)(x)=(P_{\mc S_{\mc D}(\msc A)}P_{\mc S_{\mc D}(\msc A')}f)(x)=P_{\mc S_{\mc D}(\msc A')}(x)P_{\mc S_{\mc D}(\msc A)}(x)f(x)
= P_{\mc S_{\mc D}(\msc A)(x)}|_{\mc S_{\mc D}(\msc A')(x)}f(x).$$  
and by  \cite[ Theorem 4.1 (iii)]{BownikIversion2021}, we have 
$$\mbox{inf}\left\{\frac{\|P_{\mc S_{\mc D}(\msc A')}f\|}{\|f\|}: f\in\mc S_{\mc D}(\msc A)\backslash \{0\}\right\}=\essinf_{x \in X}\left\{\frac{\|P_{\mc S_{\mc D}(\msc A')(x)}w\|}{\|w\|}: w\in J_{\msc A}(x)\backslash \{0\}\right\}.$$
Thus  if we define,  $\sigma (\mc S_{\mc D}(\msc A)):= \{x \in X:  J_{\msc A}(x) \neq 0 \} $ then
\begin{align}\label{Pointwise-Inf}
		R(\mc S_{\mc D}(\msc A), \mc S_{\mc D}(\msc A'))= \begin{cases}
			\essinf_{x\in \sigma( \mc S_{\mc D}(\msc A))} R(J_{\msc A}(x), J_{\msc A'}(x)) \  \mbox{if} \  \mu_X( \sigma( \mc S_{\mc D}(\msc A)))>0,\\
			1, \mbox{otherwise}.
		\end{cases}
. 
\end{align}

\begin{prop}\label{ObliqueDecomp}
	 In addition to  the assumptions of Proposition \ref{BesselCharrecterisation}, 
%	 let us assume $\mc E_{\mc D}(\msc A)$ and  $\mc E_{\mc D}(\msc A')$  be   frames for $\mc S_{\mc D}(\msc A)$  and $\mc S_{\mc D}(\msc A')$, respectively.  
	  the  following  statements are equivalent:
	\begin{itemize}
	\item[(i)] For a.e. $x \in X$,  the relation (\ref{rank}) holds, i.e.,	$\rank \  G_{\msc A, \msc A'}(x)=\dim J_{\msc A} (x)=\dim J_{\msc A'} (x), \  a.e. \ x\in X,$  and      there exists a constant $C>0$ such that 
 $$\|(G_{\msc A}(x))^{1/2}G_{\msc A, \msc A
 }(x)^\dagger(G_{\msc A'}(x))^{1/2}\| \leq C, \  a.e. \ x \in \{x \in X:  J_{\msc A}(x) \neq 0 \} := \sigma (\mc S_{\mc D}(\msc A)),
	$$
	where $G_{\msc A, \msc A'}(x)^\dagger$ denotes the pseudo inverse of $G_{\msc A, \msc A'}(x)$. 
		\item[(ii)] $L^2(X;
		\mc H)=\mc S_{\mc D}(\msc A)\oplus \mc S_{\mc D}(\msc A')^\perp$. 
		\item[(ii)] $L^2(X;
		\mc H)=\mc S_{\mc D}(\msc A')\oplus \mc S_{\mc D}(\msc A)^\perp$. 
		\item[(iv)] $R(\mc S_{\mc D}(\msc A), \mc S_{\mc D}(\msc A'))>0$ and $R(\mc S_{\mc D}(\msc A'), \mc S_{\mc D}(\msc A))>0$  
	\end{itemize}
\end{prop}
\begin{proof} The result can be establish easily following the steps of   \cite[Theorem 4.18]{BownikIversion2021}  and  \cite[Theorem 3.8]{kim2005infimum}.
	\end{proof}

At the end of this section we provide a method to construct alternate (oblique) duals, which is  an abstraction version of \cite[Lemma 5.1]{kim2005infimum}.
\begin{prop}\label{Oblique-Dual-Construction}  For a  
	   $\sigma$-finite measure space  $(X, \mu_{X})$      with $\mu (X)<\infty$,  consider the assumptions of Proposition \ref{BesselCharrecterisation} and  assume $\mc E_{\mc D}(\msc A)$ to be a frame for $\mc S_{\mc D}(\msc A).$ 
Define a class of functions $\tilde{\msc A}'=\{h_i\}_{i=1}^r$  associated to  $\msc A'=\{g_i\}_{i=1}^r  \subset L^2(X;\mc H)$ by 
\begin{align}\label{dualisation}
	h_{i}(x)=\begin{cases}
		\sum_{j=1}^{r}\overline{G_{\msc A, \msc A'}(x)^\dagger_{i,j}}\,  g_j(x),  \ \mbox{if } x\in \sigma (\mc S_{\mc D}(\msc A)),\\
		0, \quad   otherwise .\  
		\end{cases}
\end{align}
Then,  $\mc E_{\mc D}(\tilde{\msc A}')$ is  an alternate (oblique) MG-dual for $\mc E_{\mc D}(\msc A)$ if the Proposition \ref{ObliqueDecomp} (i) rank condition holds and there exists a $C>0$ such that $\|G_{\msc A,\msc A'}(x)^\dagger\|\leq C$	 a.e. $x\in \sigma (\mc S_{\mc D}(\msc A))$.
\end{prop}
\begin{proof}  Note that  $G_{\tilde{\msc A'}}(x)= 
G_{\msc A, \msc A'}(x)^\dagger G_{\msc A'}(x)(G_{\msc A, \msc A'}(x)^\dagger)^*,$   a.e. $x \in \sigma(\mc S_{\mc D}(\msc A))$ and $G_{\tilde{\msc A'}}(x)= 
0$,  otherwise,       $\|G_{\tilde{\msc A'}}(x)\|$ is bounded above  due to Bessel property of $\msc A' (x)$. By the Proposition \ref{AlternateDual} we need to verify $G_{\msc A}(x)G_{\msc A, {\tilde{ \msc A'}}}(x)=G_{\msc A}(x)$  which follows from  the same technique of proof  \cite[Lemma 5.3]{kim2005infimum}. 
\end{proof}

\section{{\bf Proof of Theorem \ref{Result-I} and  Theorem \ref{Result-II}}}\label{S: Proof}

\begin{proof}[Proof of Theorem \ref{Result-I}]
(ii) $\ra$ (i): Assume that (ii) holds, then we have  $\rank \ G_{\msc A,\msc B}(x)=\dim J_{\msc A}(x)=\dim J_{\msc B}(x)$ a.e. $x\in X$ by Proposition \ref{ObliqueDecomp} (iv). Considering the  projection   $P(x):=P_{J_{\msc A}(x)|_{\msc B(x)}}: J_{\msc B(x)}\ra J_{\msc A(x)}$, we have  $G_{\msc A,\msc B}(x)= T_{\msc B}(x) P(x)T_{\msc A}^*(x)$,  and $P(x)$ is invertible by  \cite[Lemma 3.1]{kim2005infimum}.  Then the length of $\mc S_{\mc D}(\msc A)= \mbox{length} \ \mc S_{\mc D}(\msc B)$ since $\mc S_{\mc D}(\msc A)$ and $\mc S_{\mc D}(\msc  B)$ are finitely generated. Let  $r$   be  the common length of $\mc S_{\mc D}(\msc A)$ and  $\mc S_{\mc D}(\msc B)$.  Then using \cite[Theorem 2.6]{bownik2015structure} there exists $\msc A^\#=\{f_i^\#\}_{i=1}^r$ and $\msc  K=\{g_i^\#\}_{i=1}^r$ such that $\mc S_{\mc D}(\msc A^\#)=\mc S_{\mc D}(\msc A)$ and $\mc S_{\mc D}(\msc B^\#)=\mc S_{\mc D}(\msc B)$. Hence  $R(S_{\mc D}(\msc A^\#),S_{\mc D}(\msc B^\#))>0$  and  $R(S_{\mc D}(\msc B^\#),S_{\mc D}(\msc A^\#))>0$. Further  applying  Proposition \ref{ObliqueDecomp} (iv), there exists a positive constant $C$ such that 
 $\|G_{\msc A^\#}(x)^{1/2}G_{\msc A^\#, \msc B^\#}(x)^\dagger G_{\msc B^\#}(x)^{1/2}\|\leq C$ a.e. $x \in \sigma (\mc S_{\mc D}(\msc A))$. 
 
  For the class of functions $\msc A^\#=\{f_i^\#\}_{i=1}^r$, define the new class of functions $\msc A'=\{f_i'\}_{i=1}^r$ 
  by,   
\begin{align*}
{	f_i'}(x)=\sum_{j=1}^r \ol{((G_{\msc A^\#}(x)^\dagger)^{1/2})}_{i,j} f_j^\#(x), \  \mbox{for a.e. } x \in X, \quad \ \mbox{for each} \  i \in \{1,2,\dots,r\}.
\end{align*}
Applying the singular value decomposition of the positive semidefinite matrix, $G_{\msc A^\#}(x)$ for a.e. $x\in X$,   $$G_{\msc A^\#}(x)=Q(x)D(x)Q(x)^*,$$
where the diagonal entries of $D(x)$ are the non-zero eigenvalues of $G_{\msc A^\#}(x)$, and $Q(x)$ is unitary.
Also,  note that 
\begin{align*}
	\|{f_i'}(x)\|^2=(G_{\msc A^\#}(x)^\dagger)^{1/2}G_{\msc A^\#}(x)(G_{\msc A^\#}(x)^\dagger)^{1/2})_{ii}=0 \ \mbox{or} \ 1. 
\end{align*}
For each $i \in \{1,2,\dots, r\}$,  $\|{f_i'}\|^2=\int_{X} \|f_i'(x)\|^2 \dx<\mu(X)<\infty$, hence $f_i'\in L^2(X; \mc H)$. Also
\begin{align*}
	G_{\msc A'}(x)= (G_{\msc A^\#}(x)^\dagger)^{1/2} G_{\msc A^\#}(x) (G_{\msc A^\#}(x)^\dagger)^{1/2}=G_{\msc A^\#}(x)^\dagger G_{\msc A^\#}(x), \ a.e. \ x \in X. 
\end{align*}
The eigenvalues of $G_{\msc A'}(x)$ are $0$ or $1$. Thus  $\mc E_{\mc D}(\msc A')$ is a frame  for $\mc S_{\mc D}(\msc A^\#)$.
Now we will show $\mc S_{\mc D}(\msc A'
)=\mc S_{\mc D}(\msc A^\#)$. It is clear that $\mc S_{\mc D}(\msc A')(x)\subset \mc S_{\mc D}(\msc A^\#)(x)$ for a.e. $x\in X$. Also,
\begin{align*}
	\dim J_{\msc A
'}(x)=\rank  \ G_{\msc A
'}(x)=\rank  \ G_{\msc A^\#}(x)=	\dim \  J_{\msc A^\#}(x).
\end{align*}
Hence $J_{\msc A
'}(x)=J_{\msc A^\#}(x)$ a.e.  $x\in X$, i.e.,  $S_{\mc D}(\msc A')=S_{\mc D}(\msc A^\#)$\cite[Proposition 2.2 (iii)]{iverson2015subspaces}. The class $\mc E_{\mc D}(\msc A
')$  is a tight frame for $\mc S_{\mc D}(\msc A^\#)$.
 In a similar way we can  show that there exists a collection $\msc B'=\{g_i'\}_{i=1}^r$ such that  $\mc E_{\mc D}({\msc B'})$  is a tight frame for $\mc S_{\mc D}(\msc B^\#)$, and also we have 
$$G_{{\msc A'}, {  \msc B'}}(x)=(G_{\msc B^\#}(x)^\dagger)^{1/2}G_{\msc A^\#, \msc B^\#}(x)(G_{\msc A^\#}(x)^\dagger)^{1/2}.
$$
Since $P(x)$ is invertible a.e.
$G_{{\msc A'}, {\msc B'}}(x)^\dagger=(G_{ \msc A^\#}(x))^{1/2}G_{\msc A^\#, \msc B^\#}(x)^\dagger(G_{\msc B^\#}(x))^{1/2}$. Now
$\|G_{{\msc A'}, {\msc B'}}(x)^\dagger\|=\|(G_{ \msc A^\#}(x))^{1/2}G_{\msc A^\#, \msc B^\#}(x)^\dagger(G_{\msc B^\#}(x))^{1/2}\|\leq C$  a.e. $x \in \sigma (\mc S_{\mc D}(\msc A^\#))$. The result follows by 
 Proposition \ref{Oblique-Dual-Construction}, $\mc E_{\mc D}({\msc B'})$ is an oblique MG-dual for $\mc E_{\mc D}({\msc A'})$. 
 
% Now applying Proposition \ref{Oblique-Dual},  $\mc E_{\mc D}({\msc B'})$ is an oblique dual for $\mc E_{\mc D}({\msc A'})$.
 
(i)$\ra$(ii):  Define  a map $\Xi: L^2(X; \mc H)\ra \mc S_{\mc D}(\msc A)$ by $\Xi f=\sum_{i=1}^r \int_{\mc M} \langle f, M_{\phi}  g_i' \rangle M_{\phi} f_i'\dm$ . Then $\Xi$ is not necessarily an orthogonal, projection. Therefore, $L^2(X; \mc H)=\range \  \Xi \ \oplus \Ker \ \Xi=\mc S_{\mc D}(\msc A)\oplus \Ker \ \Xi$. We now show that  Ker $\Xi=\mc S_{\mc D}(\msc B)^\perp$. Let $\varphi\in \Ker \Xi$. Then $\varphi=f-\Xi f$, $f\in L^2(X;\mc H)$. For $\psi \in \mc S_{\mc D}(\msc B)$,
\begin{align*}
	\langle \varphi, \psi \rangle &= \langle f-\Xi f, \psi\rangle =\langle f, \psi\rangle -\langle \Xi f, \psi\rangle = \langle f, \psi\rangle - \sum_{i=1}^{r} \int_{\mc M} \langle f, M_{\phi} g_i'\rangle \langle M_{\phi} f_i', \psi \rangle\dm \\
&=\langle f, \psi\rangle -\left \langle f, \sum_{i=1}^r \int_{\mc M}  \langle \psi, M_{\phi}f_i' \rangle M_{\phi}  g_i'\dm \right\rangle \\
&= \langle f,\psi\rangle -\langle f, \psi\rangle =0. 
\end{align*}
Hence $\varphi \in \mc S_{\mc D}(\msc B)^\perp$. Other side, if $\varphi \in \mc S_{\mc D}(\msc B)^\perp$ then $\Xi \varphi=0$.

 (i)$\leftrightarrow$(iii):  Since $\mc E_{\mc D}(\msc  A')$ and $\mc E_{\mc D}(\msc B')$ are frames for $\mc S_{\mc D}(\msc A)$ and $\mc S_{\mc D}(\msc B)$, respectively, the  systems  $\msc A'(x)=\{f_i'(x):i=1,\dots, r\}$ and $\msc B'(x)=\{g_i'(x):i=1,\dots, r\}$   are frames for $J_{\msc A}(x)$ and $J_{\msc B}(x)$ respectively \cite[Theorem 2.10]{iverson2015subspaces},    for a.e. $x\in X$.  The rest part of the result follows by Proposition \ref{AlternateDual}. 

 In a similar ways, the converse part follows.

(iv)$\ra$(iii): Assume (iv) holds, i.e., there exist  frames $\{ f_i'(x)\}_{i=1}^r$ and  $\{g_i'(x)\}_{i=1}^r$   for $J_{\msc A}(x)$ and $J_{\msc B}(x)$, respectively. We need to show that $R(J_{\msc A}(x),J_{\msc B}(x))>0 \ 
\mbox{and} \  R(J_{\msc B}(x),J_{\msc A}(x))>0$, which is equivalent to $J_{\msc A}(x)\oplus J_{\msc B}(x)^\perp=\mc H$  \cite[Lemma 2.1]{christensen2004oblique}. For this define a map, 
$\Xi: \mc H\ra J_{\msc A}(x)$ by 
$\Xi(f)=\sum_{i=1}^r \langle f, g_i'(x)\rangle f_i'(x)$. Then $\Xi$ need not be an orthogonal projection. Hence $$\mc H=\range \Xi \ \oplus \  \Ker \Xi= J_{\msc A}(x)\oplus \Ker \Xi.$$ 
Our aim to prove  $\Ker \Xi=J_{\msc B}(x)^\perp$. Let $u\in \Ker \Xi$. Then $u=f-\Xi f$ for some $f$. Let $h\in J_{\msc B}(x)$. Writing
\begin{align*}
	\langle u, h\rangle =\langle f-\Xi f, h\rangle=\langle f,h \rangle -\left \langle \sum_{i=1}^r\langle f, g_i'(x)\rangle f_i'(x),h \right\rangle=\langle f,h \rangle -\sum_{i=1}^r \langle  f, g_i'(x)\rangle \langle f_i'(x),h\rangle\\=\langle f,h \rangle -\left\langle f,\sum_{i=1}^r  \langle h, f_i'(x)\rangle g_i'(x)\right \rangle=\langle f, h\rangle -\langle f, h\rangle=0, 
\end{align*}
we have $ u\in J_{\msc B}(x)^\perp,$ and if $u\in J_{\msc B}(x)^\perp$, 
then $u\in \Ker \Xi$.  

(ii)$\ra$ (iv): If $R(\mc S_{\mc D}(\msc A),\mc S_{\mc D}(\msc B))>0$, we have  $R(J_{\msc A}(x),J_{\msc B}(x))>0 $  for a.e. $x\in  X$.  The remaining part follows easily.
\end{proof}

Before moving towards the proof of Theorem \ref{Result-II} we need the concept of  supremum cosine angle. For two subspaces $ V$ and $ W$ of a  Hilbert space $\mc H$, the supremum cosine angle between them is: 
$S( V,  W)= \sup_{v\in  V \backslash \{0\}} \|P_{ W} v\|/\|v\|$. The correlation  between supremum and  infimum cosine angle is related with the following: $R(V, W)=\sqrt {1-S(V, W^\perp)^2}$. One of the main uses of supremum cosine angle is  to determine an addition of two closed subspaces is again closed or not.  The sum of two closed subspaces $V$ and $W$ is again closed and $V\bigcap W=\{0\}$ if and only if $S(V, W)<1$ \cite[Theorem 2.1]{Tang2000obliqueprojection}.

\begin{proof}[Proof of the Theorem \ref{Result-II}]
(i) {\bf Global Setup:}
Suppose $\mc E_{\msc D}(\msc A)$ and $\mc E_{\msc D}(\msc A')$ are Riesz basis for $\msc V$ and $\msc W$, with constants $A,B$ and $ A',  B'$ respectively,  and  are biorthogonal. By   \cite[ Theorem 2.3]{iverson2015subspaces} we have $\msc A(x)=\{f_i(x):i=1,2,\dots, r\}$ and $\msc A'(x)=\{f'_i(x):i=1,2,\dots, r\}$ are    Riesz bases for $J_{\msc A}(x)$ and $J_{\msc A'}(x)$  for a.e. $x \in X$.  It suffices to show $R(\msc V,\msc W)=R(\msc W, \msc V)>0.$ The dual Riesz basis for $\mc E_{\msc D}(\msc A)$ in $\msc V$ is of the form $\mc E_{\msc D}(\msc A^{\#})$ where $\msc A^{\#}=\{f_i^{\#}:i=1,2,\dots,r\}\ss \msc V$. Therefore the orthogonal projection $P_{\msc V}$ onto $\msc V$ can be expressed as 
$$
P_{\msc V}f=\sum_{i=1}^r \sum_{\phi \in \msc D}\langle f, M_{\phi}f_i^{\#}\rangle M_{\phi} f_i=\sum_{i=1}^r \sum_{\phi \in \msc D}\langle f, M_{\phi}f_i^{}\rangle M_{\phi} f_i^{\#}, \  \forall  \ f\in L^2(X;\mc H).
$$
Observe that $P_{\msc V} M_{\phi} f_i'=f_i^{\#}$,  for all $\phi \in \msc D$ and $i=1,2,\dots, r$. For $f\in \msc W\backslash \{0\}$, 
we have $f=\sum_{i=1}^r \sum_{\phi \in \msc D}c_i^{\phi} M_{\phi}f_i'$,  where 
$$A' \sum_{i=1}^r\sum_{\phi \in \msc D}|c_{i}^\phi|^2\leq \|f\|^2 \leq B' \sum_{i=1}^r\sum_{\phi \in \msc D}|c_{i}^\phi|^2. $$
Then $P_{\msc V}f=\sum_{i=1}^r \sum_{\phi \in \msc D} c_{i}^\phi M_\phi f_i^{\#}$  and $\frac{\|P_{\msc V}f\|^2}{\|f\|^2}\geq \frac{B^{-1}\sum_{i,\phi}|c_{i}^\phi|^2}{B'\sum_{i, \phi}|c_{i}^\phi|^2}=\frac{1}{BB'}$,
since $\mc E_{\msc D}(\msc A^\#)$ is a Riesz basis with constants $B^{-1},  A^{-1}$. Hence  $R(\msc V, \msc W)\geq (BB')^{-1/2}$.

 Conversely, Since $R(\msc  W, \msc V)>0$, then $R(\msc W, \msc V)\|f\| \leq \|f\|\leq \|f\|, \forall f\in \msc V$. Since $\mc E_{\msc D}(\msc A)$ is a Riesz basis for $\msc V$,  then the corresponding projection on $ \msc W$,  that is,   $\{P_{\msc W}M_{\phi}f_i:\phi \in \msc D, i=1,2,\dots,r\}$ is a Riesz basis for $\msc W$.  Since $R(\msc V,\msc W)>0$, we get   $\ol{\mbox{span}}\{P_{\msc W}M_{\phi}f_i:\phi \in \msc D, i=1,2,\dots,r\}=\msc W$, and by \cite[Corrollary 5.14]{BownikIversion2021} there exists a  dual Riesz basis for $\ol{\mbox{span}}\{P_{\msc W}M_{\phi}f_i:\phi \in \msc D, i=1,2,\dots,r\}$ of the multiplication generated form, i.e.,  $\{M_{\phi}f_i':\phi\in \msc D, i=1,2,\dots,r\}$ in $\msc W$. Thus we have 
$$\langle M_{\phi}f_j, M_{\phi'}f_i'\rangle =\langle M_{\phi}f_j,P_{\msc W} M_{\phi'}f_i' \rangle=\langle P_{\msc W} M_{\phi}f_j, M_{\phi'}f_i' \rangle=\delta_{\phi, \phi'}\delta_{i,j} \  \phi, \phi'\in \msc D, \mbox{and} \  i, j \in \{1,2,\dots,r\}.$$
Thus the result follows. 

(ii) {\bf Local  Setup:}  For a.e. $x\in X$,  let $\{f_i(x)\}_{i=1}^r$ and $\{f_i'(x)\}_{i=1}^r$ be  Riesz bases for $J_{\msc A}(x)$ and $J_{\msc A'}(x)$, respectively and they are biorthogonal. We now show this  is equivalent to 
\begin{align}\label{DSome-J}
\ J_{\msc A'}(x)\oplus J_{\msc A}(x)^\perp=\mc H \  \mbox{and} \ 	J_{\msc A}(x)\oplus J_{\msc A'}(x)^\perp=\mc H, \ \mbox{for a.e.} \ x\in X.
\end{align}
Since $\{f_i(x)\}_{i=1}^r$ and $\{f_i'(x)\}_{i=1}^r$ are Riesz basis, then $J_{\msc A}(x)=\{u\in \mc H: u=\sum_{i=1}^r c_i f_i(x)\}$ and $J_{\msc A'}(x)=\{v\in \mc H: v=\sum_{i=1}^r c_i f_i'(x)\}$.  Let $h\in J_{\msc A}(x)\bigcap J_{\msc A'}(x)^\perp$ then 
$h=\sum_{i=1}^r \langle h, f_i(x)\rangle f_i'(x)=0$ hence $J_{\msc A'}(x)\bigcap J_{\msc A}(x)^\perp=\{0\}$. 
 Let $w\in \mc H$, then $Pw:=\sum_{i=1}^r \langle w, f_i(x)\rangle f_i'(x)\in J_{\msc A'}(x)$.   By the biorthogonal property of  $\{f_i(x)\}_{i=1}^r$ and $\{f_i'(x)\}_{i=1}^r$, we have $\langle w-Pw, f_i(x) \rangle=0$ for all $i=1,\dots,r$ i.e., $w-Pw\in J_{\msc A}(x)^\perp$.
So $w=Pw+ (w-Pw)\in J_{\msc A'}(x)+ J_{\msc A}(x)^\perp $ which implies $J_{\msc A'}(x)+J_{\msc A}(x)^\perp=\mc H$. Combining, we have $	J_{\msc A'}(x)\oplus J_{\msc A}(x)^\perp=\mc H $. In a similar way, interchanging the roll of $J_{\msc A}(x)$ and $J_{\msc A'}(x)$ other side of (\ref{DSome-J}) i.e., 	$J_{\msc A}(x)\oplus J_{\msc A'}(x)^\perp=\mc H, \ \mbox{for a.e.} \ x\in X$, can be shown.

Since  $J_{\msc A}(x)+ J_{\msc A'}(x)^\perp$ is closed and $J_{\msc A}(x)\bigcap J_{\msc A'}(x)^\perp=\{0\}$, by the  \cite[Theorem 2.1]{Tang2000obliqueprojection} the supremum cosine angle
$$S(J_{\msc A}(x), J_{\msc A'}(x)^\perp):=\sup \{|\langle v, w\rangle|:v\in J_{\msc A}(x), w\in J_{\msc A'}(x)^\perp, \|v\|=\|w\|=1\}<1.$$
Hence $R(J_{\msc A}(x), J_{\msc A'}(x) )=\sqrt{1-S(J_{\msc A}(x), J_{\msc A'}(x)^\perp)^2}>0$. Interchanging the roll of $J_{\msc A}(x)$ and $J_{\msc A'}(x)$ in the above argument $R(J_{\msc A'}(x), J_{\msc A}(x))>0$.

 Converse part,
Let $R(J_{\msc A}(x), J_{\msc A'}(x))>0$. Then $S(J_{\msc A}(x), J_{\msc A'}(x)^\perp)<1$. Using \cite[Theorem 2.1]{Tang2000obliqueprojection}, we have $J_{\msc A}(x)+J_{\msc A'}(x)^\perp$ is closed and $J_{\msc A}(x)\bigcap J_{\msc A'}(x)^\perp=\{0\}$. In a similar way when $R(J_{\msc A'}(x), J_{\msc A(x)})>0$, then we can show $J_{\msc A'}(x)+J_{\msc A}(x)^\perp$ is closed and $J_{\msc A'}(x)\bigcap J_{\msc A}(x)^\perp=\{0\}$.
 Hence 
$$J_{\msc A}(x)+J_{\msc A'}(x)^\perp=(J_{\msc A}(x)+J_{\msc A'}(x)^\perp)^{\perp \perp}=\left(J_{\msc A}(x)^\perp\bigcap J_{\msc A'}(x)\right)^\perp=\mc H.$$ So
$\mc H=J_{\msc A}(x) \oplus J_{\msc A'}(x)^\perp$.  In a similar way, $\mc H=J_{\msc A'}(x) \oplus J_{\msc A}(x)^\perp$. 

Assume $J_{\msc A'}(x)\bigoplus J_{\msc A}(x)^\perp=\mc H$. Let for a.e. $x\in X$, $\{f_i(x)\}_{i=1}^r$ and $\{h_i(x)\}_{i=1}^r$  be the dual Riesz bases for $J_{\msc A'}(x)$, i.e., 
$\langle f_i(x), h_j(x)\rangle=\delta_{i,j}$. Let $S: J_{\msc A}(x)\ra J_{\msc A}(x)$ be the frame operator, then consider $$g_i(x):={S^{-1}f_i}(x), 1\leq i\leq r.$$ 
Now the map, $P_{J_{\msc A(x)}}: \mc  H\ra \mc H$ by $P_{J_{\msc A}(x)}f:=\sum_{i=1}^r \langle f, f_i(x) \rangle g_i(x)$ is
the orthogonal projection of $\mc H$ on $J_{\msc A}(x)$.
Consider $$\msc P:=P_{J_{\msc A}(x)}|_{J_{\msc A'}(x)}.$$  If $f\in J_{\msc A'}(x)$
 then $\msc P(f)=0$ then $f\in J_{\msc A'}(x)\bigcap J_{\msc A}(x)^\perp=\{0\}$ so $\msc P$ is injective and $\msc P(J_{\msc A'}(x))=P_{J_{\msc A}(x)}(J_{\msc A'}(x))=P_{J_{\msc A}(x)}(J_{\msc A'}(x)+J_{\msc A}(x)^\perp)=P_{J_{\msc A}(x)}(\mc H)=J_{\msc A}(x)$. Hence $\msc P$ is  bounded invertible operator. Define $f_i'(x):=\msc P^{-1}(f_i(x))$, $1\leq i\leq r$. Then $\{f_i'(x)\}_{i=1}^r$ is  the required Riesz basis for $J_{\msc A'}(x)$,  satisfying the biorthogonality condition.

\end{proof}

\section{Application to locally compact group}\label{Application}
Let $\mathscr G$ be a second countable locally compact group which is not necessarily abelian, and  $\Gamma$ be a closed abelian subgroup of $\mathscr G$. A closed subspace  $V$ in $L^2(\mathscr G)$ is said to be   \textit{$\Gamma$-translation invariant ($\Gamma$-TI)} if $L_\xi f \in  V$  for all $f\in  V$ and $\xi \in \Gamma$, where for $\eta \in \mathscr G$ the \textit{left translation} $L_\eta$ on $L^2 (\mathscr G)$ is defined by    
	$$(L_\eta f)(\gamma) = f(\eta^{-1} \gamma), \quad \gamma \in \mathscr G \  \mbox{and }\ f\in L^2(\mathscr G).$$
	Translation invariant spaces are widely used in various domains, significant among them are harmonic analysis, signal processing, and time-frequency analysis.
	Researchers are  often interested in characterizing the class of  generators of TI spaces  that allows for reconstruction of any function/signal/image \textit{via} a reproducing formula. 
	For a family of functions   $\mc A\ss L^2(\mathscr G)$, 
	let us consider a \textit{$\Gamma$-translation generated ($\Gamma$-TG)} system $\mc E^{\Gamma}(\mc A)$ and   its associated  $\Gamma$-translation invariant ($\Gamma$-TI)  space $\mc S^{\Gamma}(\mc A)$  generated by $\mc A$, i.e., 
	$$
	\mc E^{\Gamma}(\mc A) :=\{L_\xi \varphi : \varphi \in \mc A, \xi\in \Gamma\}  \quad \mbox{and} \quad  \mc S^{\Gamma}(\mc A) := \overline{\Span} \ \mc E^{\Gamma}(\mc A),
	$$ 
	respectively.

For  $x \in  \mathscr G$, a right coset of $\Gamma$ in $\mathscr G$ with respect to $x$ is denoted by $\Gamma x$, and for  a function   $f :\mathscr G \rightarrow \mathbb C$, we define a  complex valued    function $f^{\Gamma x}$ on $\Gamma$ by  
$
f^{ \Gamma x}(\gamma)=f(\gamma \, \Xi(\Gamma x)), \quad     \gamma \in  \Gamma, 
$
where   the   space of orbits        $\Gamma\backslash \mathscr G = \{\Gamma x: x \in  \mathscr G\}$ is    the set of all right cosets of $\Gamma $ in $\mathscr G$, and $\Xi : \Gamma \backslash \mathscr G\ra  \mathscr G$ is a \textit{Borel section} for  the quotient space $\Gamma \backslash \mathscr G$. Then the Fourier transform of $f^{ \Gamma x} \in L^1 (\Gamma)$  is given by $\widehat{f^{ \Gamma x}} (\alpha)=\int_{\Gamma} f^{ \Gamma x} (\gamma) \alpha (\gamma^{-1}) d\mu_\Gamma (\gamma),$ for $\alpha \in \widehat{\Gamma},$ 
which can be  extended to $L^2 (\Gamma)$.  The \textit{Zak transformation} $\mathcal Z$ of  $f\in L^2(\mathscr G)$ for the pair $(\mathscr G,\Gamma)$  is   defined by
\begin{equation}\label{zak trans}
	(\mc Zf)(\alpha)(\Gamma x) =\widehat{f^{\Gamma x}}(\alpha), \quad a.e. \quad  \alpha \in \widehat{\Gamma} \ \mbox{and} \   \Gamma x\in \Gamma\backslash \mathscr G,
\end{equation}
which is a    unitary linear transformation from $ L^2(\mathscr G)  $ to $L^2(\widehat{\Gamma}; L^2(\Gamma\backslash \mathscr G))$ \cite{ iverson2015subspaces}.  
Note that the Zak transform $\mc Z$ is closely associated to fiberization map $\mathscr F$ when $\mathscr G$ becomes   abelian. For a second countable LCA group $\mc G$ and its closed subgroup $\Lambda$,  the \textit{fiberization}   $\mathscr F$ is a unitary map from $L^2 (\mc G)$ to $L^2 ( \widehat{\mc G}/ \Lambda^\perp ; L^2 (\Lambda^\perp))$ given by 
$  (\mathscr F f)(\beta\Lambda^\perp )(x)=\widehat{f} (x \, \zeta (\beta \Lambda^\perp)),$ $x\in \Lambda^\perp, \beta\in \widehat{\mc G}, $
for $f \in L^2 (\mc G)$, where $\Lambda^\perp:=\{\beta\in \widehat{\mc G}: \beta(\lambda)=1, \ \forall \ \lambda \in \Lambda \}$, $\Lambda^\perp\backslash \widehat{\mc G}=\widehat{\mc G}/ \Lambda^\perp$ and  $\zeta : \widehat{\mc G}/ \Lambda^\perp \rightarrow \widehat{\mc G}$ is Borel section which   maps compact sets to pre-compact sets. 	 
 The  Zak transform and fiberization map  on the Euclidean space $\mathbb R^n$  by the action of integers $\mathbb Z^n$ are 
 $$
({\mc Z} f) (\xi, \eta)=\sum_{k \in \mathbb Z^n} f(\xi+k) e^{-2\pi i k\eta}, \ \mbox{and}  \  (\mathscr Ff)(\xi)(k)=\widehat {f}(\xi+k),  
 $$
 for $k \in \mathbb Z^n$, $\xi, \eta \in \mathbb  T^n$ and $f \in L^1 (\mathbb R^n) \bigcap L^2 (\mathbb R^n)$.

Observe that    the Zak transform $\mc Z$   satisfies the intertwining property with the left translation and multiplication operators,  i.e., for $f \in  L^2 (\mathscr G)$, 
$	(\mc ZL_{\gamma}f)(\alpha) =  (M_{\phi_\gamma}  \mc Zf) (\alpha), \quad 
	\mbox{for    a.e.} \ \alpha \in \widehat{\Gamma}\  \mbox{and} \ \gamma \in \Gamma, $ 
where  $M_{\phi_\gamma}$ is the multiplication operator on $L^2(\widehat{\Gamma}; L^2(\Gamma\backslash \mathscr G))$, $\phi_\gamma (\alpha)=\overline{\alpha(\gamma)}$   and $\phi_\gamma \in L^\infty (\widehat{\Gamma})$ for each $\gamma \in \Gamma$. Therefore, our goal can be established by converting the problem of $\Gamma$-TI space $\mc S^{\Gamma}(\msc A)$ into the MI spaces on $L^2 (X; \mc H)$ with the help of Zak transform, where $X=\widehat{\Gamma}$ and $\mc H=L^2(\Gamma\backslash \mathscr G)$.

%Similar situation will arise   for the case of   abelian group $\mc G$ and its closed subgroup $\Lambda$ since the fiberization map  $\mathscr F$ satisfies the following intertwining relation  for $f \in L^2 (\mc G)$:
%\begin{equation*}
%	(\mathscr F L_{\lambda} f)(\beta\Lambda^\perp ) = (M_{\phi_\lambda}  \mathscr Ff)  (\beta\Lambda^\perp )  \quad \mbox{for  a.e.} \ \beta \in \widehat{\mc G} \ \mbox{and} \  \lambda \in \Lambda,  
%\end{equation*}
%where  $\phi_\lambda \in L^\infty ( \widehat{\mc G}/ \Lambda^\perp)$ given by $\phi_\lambda (\beta\Lambda^\perp)=\overline{\beta(\lambda)}$ and the multiplication operator $M_{\phi_\lambda}$ is defined on $L^2 ( \widehat{\mc G}/ \Lambda^\perp ; L^2 (\Lambda^\perp))$.

The following result is a generalization of \cite[Theorem  ]{kim2005infimum} for the locally compact group.
\begin{thm}
Let $\msc G$ be a locally compact group having a discrete abelian subgroup $\Gamma$, then 
	for the finite collection of functions ${\mc A}=\{f_i\}_{i=1}^m$ and ${\mc B}=\{g_i\}_{i=1}^n$ in $L^2(\msc G)$,  and for a.e. $\alpha \in \widehat \Gamma$, assume the range functions  $J_{\mc A}(\alpha)=\Span\{\mc Zf_i(\alpha): i=1,2, \dots,m \}$ and $J_{\mc B}(\alpha)=\Span\{\mc Zg_i(\alpha): i=1,2, \dots,n \}$ associated to $\Gamma$-TI spaces $\mc S^{\Gamma}(\mc A)$ and $\mc S^{\Gamma}(\mc B)$, respectively. Then the following are equivalent:
	\begin{enumerate}
		\item[(i)]  	 There exists ${\mc A'}=\{f_i'\}_{i=1}^r$ and ${\mc B'}=\{g_i'\}_{i=1}^r$ in $L^2(\msc G)$ such that  $\mc E^{\Gamma}(\mc A')$ and  $\mc E^{\Gamma}(\mc B')$ are continuous frames for $\mc S^{\Gamma}(\mc A)$  and $\mc S^{\Gamma}(\mc B)$, respectively, satisfying 
		the  following reproducing formulas for  $g\in \mc S^{\Gamma}(\mc A)$ and $h \in \mc S^{\Gamma}(\mc B)$:
		\begin{align*}
			g=\sum_{i=1}^r \int_{\Gamma} \langle g, L_\gamma g_i'\rangle L_\gamma f_i' \dgamma\  ,\,  \mbox{and}\   h=\sum_{i=1}^r  \int_{\Gamma}\langle h, L_{\gamma}f_i'\rangle L_{\gamma} g_i'  \ \dgamma.
		\end{align*}
		\item[(ii)] The infimum cosine angles  of  $\mc S^{\Gamma}(\mc A)$  and $\mc S^{\Gamma}(\mc B)$  are greater than zero, i.e.,  $$
		R(\mc S^{\Gamma}(\mc A) ,\mc S^{\Gamma}(\mc B) )>0 \ 
		\mbox{and} \  R(\mc S^{\Gamma}(\mc B) ,\mc S^{\Gamma}(\mc A) )>0.$$
		\item[(iii)]  There exists collection of functions  $\{f_i'\}_{i=1}^r$ and $\{g_i'\}_{i=1}^r$ in $L^2(\msc G)$ such that for a.e. $\alpha \in \widehat \Gamma$, the systems $\{\mc Zf_i'(\alpha)\}_{i=1}^r$ and  $\{\mc Z g_i'(\alpha)\}_{i=1}^r$ are finite frames for $J_{\msc A}(\alpha)$  and $J_{\msc B}(\alpha)$, respectively, satisfying 
		the  following reproducing formulas  for   $u\in J_{\mc A}(\alpha)$ and   $v\in J_{\mc B}(\alpha)$:  
		\begin{align*}
			u=\sum_{i=1}^r \langle u,\mc Z g_i'(\alpha)\rangle \mc Zf_i'(\alpha), \ \mbox{and} \ v=\sum_{i=1}^r \langle v, \mc Zf_i'(\alpha)\rangle \mc Zg_i'(\alpha), \  a.e. \ \alpha \in \widehat \Gamma.
		\end{align*}
		
		\item[(iv)]  For a.e. $\alpha \in \widehat \Gamma$, the infimum cosine angles of  $J_{\mc A}(\alpha)$ and 
		$J_{\mc B}(\alpha)$  are greater than zero, i.e.,
		$$R(J_{\mc A}(\alpha),J_{\mc B}(\alpha))>0 \ 
		\mbox{and} \  R(J_{\mc B}(\alpha),J_{\mc A}(\alpha))>0.$$
	\end{enumerate}
\end{thm}
\begin{proof}
Since $\mc Z$ is an unitary operator,  	$R(\mc S^{\Gamma}(\mc A) ,\mc S^{\Gamma}(\mc A'))=R(\mc Z\mc S^{\Gamma}(\mc A) ,\mc Z\mc S^{\Gamma}(\mc A'))$ using Definition \ref{D-InfCosine}. Hence we have the desired result   using Theorem \ref{Result-I}.
\end{proof}

Next  we state the following result which is a generalization  to the locally compact group in case of Riesz basis \cite[Proposition 2.13]{bownik2004biorthogonal}.

\begin{thm}
	
Let $\msc G$ be a locally compact group having a discrete abelian subgroup $\Gamma$ and $\msc V$, $\msc W$ be $\Gamma$-TI subspaces of $L^2(\mathscr G)$. For the finite collection of functions ${\msc A}=\{f_i\}_{i=1}^r$,  assume $\mc E^{\Gamma} (\mc A)$ is a Riesz basis for $\msc V$.
	Then the following holds:
	\begin{enumerate}
		\item[(i)] \textbf{Global setup:} If  there exists  ${\mc A'}=\{f'_i\}_{i=1}^r$  in $L^2(\msc G)$ such that $\mc E^\Gamma(\mc A')$ is a Riesz basis for $\msc W$ satisfying the  biorthogonality  condition 
	$ \langle L_{\gamma} f_i, L_{\gamma'}f'_{i'} \rangle=\delta_{i, i'}\delta_{\gamma, \gamma'}, \quad i, i'=1, 2, \cdots, r; \  \gamma, \gamma'\in \Gamma, 
		$
		then 				the infimum cosine angles  of $\msc V$ and $\msc W$ are greater than zero, i.e.,  
		\begin{align}\label{Infimum1}
			R(\msc V,\msc W)>0 \ 
			\mbox{and} \  R(\msc W, \msc V)>0. 
		\end{align}
		Conversely	if  (\ref{Infimum1})  holds true, then  there exists   ${\mc A'}=\{f'_i\}_{i=1}^r$  in $L^2(\msc G)$ such that $\mc E^\Gamma(\mc A')$ is a Riesz basis for $\msc W$ satisfying the  biorthogonality  condition.  
		Moreover, the following reproducing formulas hold: 
		$$f=\sum_{\gamma \in \Gamma} \sum_{i=1}^r\langle f, L_{\gamma}f_i'\rangle L_{\gamma'}f_i,  \  \forall f\in \msc V,  \   \mbox{and} \  g=\sum_{\gamma \in \Gamma} \sum_{i=1}^r \langle g, L_{\gamma}f_i\rangle L_{\gamma}f_i',  \  \forall g\in \msc W.$$
		
		\item[(ii)]  \textbf{Local setup:} If  there exists  ${\mc A'}=\{f'_i\}_{i=1}^r$  in $L^2(\msc G)$  such that for a.e. $\alpha \in \widehat \Gamma$, $\{\mc Zf'_i(\alpha)\}_{i=1}^r$  is a  Riesz sequence in $L^2(\Gamma \backslash \msc G)$   satisfying the  following biorthogonality  condition 
		\begin{align}\label{Localbiorthogonal1}
			\langle   \mc Zf_i(\alpha),  \mc Zf'_{i'}(\alpha) \rangle=\delta_{i, i'}, \quad i, i'=1, 2, \cdots, r,    \ a.e. \  \alpha \in \widehat \Gamma, 
		\end{align}
		the infimum cosine angles  of $J_{\mc A}(\alpha)=\Span \{\mc Zf_i(\alpha)\}_{i=1}^r$ and $J_{\mc A'}(\alpha)=\Span \{\mc Zf_i'(\alpha)\}_{i=1}^r$ are greater than zero, i.e.,  
		\begin{align}\label{LocalInfimum1}
			R(J_{\mc A}(\alpha), J_{\mc A'}(\alpha))>0 \ 
			\mbox{and} \  R(J_{\mc A'}(\alpha), J_{\mc A}(\alpha))>0, \ a.e. \alpha \in \widehat \Gamma. 
		\end{align}
		Conversely if (\ref{LocalInfimum1}) holds,      there exists   ${\mc A'}=\{f'_i\}_{i=1}^r$  in $L^2(\msc G)$ such that  for  a.e. $\alpha \in \widehat \Gamma$,  $\{\mc Zf_i'(\alpha)\}_{i=1}^r$ is a Riesz sequence in  $L^2(\Gamma\backslash \msc G)$ satisfying the  biorthogonality  condition (\ref{Localbiorthogonal1}).  Moreover,  the following reproducing formulas hold for $u\in J_{\mc A}(\alpha)$,   and $v\in J_{\mc A'}(\alpha)$:  
		\begin{align*} 
			u=\sum_{i=1}^r \langle u, \mc Zf_i'(\alpha)\rangle \mc Zf_i(\alpha), \ \mbox{and} \ v=\sum_{i=1}^r \langle v, \mc Zf_i(\alpha)\rangle \mc Zf_i'(\alpha), \ \mbox{for a.e.} \  \alpha \in \widehat \Gamma.
		\end{align*}
	\end{enumerate}
	\end{thm}
The similar results can be  deduced for locally compact abelian group $\mc G$ using fiberization $\msc F$.

\noindent{\bf Declarations}\\
{\bf Competing Interests} The authors have not disclosed any competing interests.\\
{\bf Data Availability} No data sets were generated during the study.
\providecommand{\bysame}{\leavevmode\hbox to3em{\hrulefill}\thinspace}
\providecommand{\MR}{\relax\ifhmode\unskip\space\fi MR }
\providecommand{\MRhref}[2]{%
	\href{http://www.ams.org/mathscinet-getitem?mr=#1}{#2}
}
\providecommand{\href}[2]{#2}

%
%\bibliographystyle{amsplain}
%\bibliography{ref2}
\end{document}